\documentclass[12pt,oneside,reqno]{amsart}

\usepackage{hyperref}

\usepackage[headsep=30pt,footskip=30pt,twoside=false,bottom=2.5cm]{geometry}

\usepackage[all]{xy}

\xymatrixcolsep{2cm}

\numberwithin{equation}{section}

\allowdisplaybreaks[1]

\usepackage{etoolbox}

\makeatletter

\patchcmd{\thesubsection}{\arabic}{\arabic}{}{}

\patchcmd{\@seccntformat}{\@secnumfont}{%
  \@secnumfont\expandafter\protect\csname format#1\endcsname}{}{}

\patchcmd{\@startsection}{\@afterindenttrue}{\@afterindentfalse}{}{}

\patchcmd{\subsection}{-.5em}{.3\linespacing}{}{}

\makeatother

\theoremstyle{plain}

\newtheorem{theorem}{Theorem}[section]

\newtheorem{proposition}[theorem]{Proposition}

\newtheorem{question}[theorem]{Question}

\newtheorem{lemma}[theorem]{Lemma}

\theoremstyle{remark}

\newcommand{\Ker}[1]{\ensuremath{\mathrm{Ker} (#1)}}

\newcommand{\cat}[1]{\ensuremath{\mathcal{#1}}}

\newcommand{\at}[2][]{\ensuremath{\mathrm{at}_{#1} (#2)}}

\newcommand{\END}[2][]{\ensuremath{\mathcal{E}\mathit{nd}_{#1} (#2)}}

\newcommand{\id}[1]{\ensuremath{\mathbf{1}_{#1}}}

\renewcommand{\dim}[2][]{\ensuremath{\mathrm{dim}_{#1}(#2)}}

\newcommand{\C}{\ensuremath{\mathbb{C}}}

\newcommand{\struct}[1]{\ensuremath{\mathcal{O}_{#1}}}

\newcommand{\HOM}[3][]{%
  \ensuremath{\mathcal{H}\mathit{om}_{#1}(#2, #3)}}

  \newcommand{\Res}[3][]{%
  \ensuremath{\mathrm{Res}_{#1}(#2, #3)}}

\newcommand{\DIFF}[4][]{\ensuremath{\mathcal{D}\mathit{iff}^{#1}_{#2}(#3,#4)}}  
  
\newcommand{\coh}[3]{\ensuremath{\mathrm{H}^{#1}(#2,#3)}}

\newcommand{\gcoh}[2]{\ensuremath{\mathrm{H}^{#1}(#2)}}

\newcommand{\drcoh}[3]{\ensuremath{\mathcal{H}}^{#1}_{#2}(#3)}

\begin{document}

\title[On the relative logarithmic connections]{ On the relative logarithmic connections and relative residue formula}

\author{Snehajit Misra and Anoop Singh}
\address{Department of Mathematics, Indian Institute of Technology, Bombay \\   Mumbai 400076, India.}
\email[Snehajit Misra]{snehajitm@math.iitb.ac.in}

\address{Tata Institute of Fundamental Research (TIFR) \\  Homi Bhabha Road, Colaba, Mumbai 400005, India.} 
\email[Anoop Singh]{anoops@math.tifr.res.in}

\subjclass[2010]{32C38, 14F10, 53C05, 14H15}
  \keywords{Relative logarithmic connection, Relative Chern class, Relative residue}

\begin{abstract}
We investigate the relative logarithmic connections on a holomorphic vector bundle over a complex analytic 
family. We give a sufficient condition for the existence of a relative logarithmic connection on a holomorphic vector bundle 
singular over a relative simple normal crossing divisor. We define the relative 
residue of relative logarithmic connection and express relative Chern classes of a holomorphic vector bundle in terms of 
relative residues.
\end{abstract}

\maketitle

\section{Introduction}
\label{Intro}
A theorem due to Atiyah \cite{A} and Weil \cite{W}, commonly known as \it Atiyah-Weil criterion\rm, says that a holomorphic vector bundle over a compact Riemann surface admits a holomorphic connection if and only if the degree of each of its indecomposable component is zero. However, it follows from \cite{A} that not every holomorphic vector bundle on a compact K\"{a}hler manifold admits a holomorphic connection. 
Although, in \cite{BS}, the relative holomorphic connections on a holomorphic vector bundle over a complex analytic family has been studied, and a sufficient condition is given for the existence of relative holomorphic connections.

Further, in \cite{BDP}, the Atiyah-Weil criterion has been generalised in the logarithmic set up, that is, 
 a necessary and sufficient condition is given for a holomorphic vector bundle on a compact Riemann surface $X$ to admit a logarithmic connection singular along a fixed reduced effective divisor $D$ on $X$ with prescribed rigid residues along $D$. There is a well studied notion 
 of relative logarithmic connection on a holomorphic vector bundle \cite{D}. 
 
Let $\pi : X \longrightarrow S$ be a complex analytic family of compact connected complex manifolds of  fixed relative dimension $l$. Let $\dim{X} = m$ and $\dim{S}=n$ so that $m = n+l$.  
%Further assume that the fibers $X_s:=\pi^{-1}(s)$ are compact connected complex submanifolds of $X$ for every $s\in S$. 
We fix  simple normal crossing (SNC) divisors $Y$ on $X$ and  $T$ on $S$ such that $\pi(Y) = T$ set-theoratically.  We say  $Y/T$ a {\bf relative SNC divisor} if the quotient sheaf 
$$\Omega^1_{X/S}(\log Y):= \frac{\Omega^1_X(\log Y)}{\pi^*\Omega^1_S(\log T)},$$
is locally free sheaf of rank $l=m-n$ on $X$,
 where $\Omega^1_X(\log Y)$ and $\Omega^1_S(\log T)$
are defined in subsection \ref{log_form} (for more details see \cite{BE}).
Note that $Y$ is a proper subset of $\pi^{-1}(T)$.

 In this article we try to answer the following question:

\begin{question}  Let $\pi:X\longrightarrow S$ be a surjective holomorphic proper submersion with connected fibers, and let $\varpi\,:\, E \,\longrightarrow\, X$ a holomorphic vector bundle. We fix a relative SNC divisor 
$Y/T$ over $X$. 
Is there a good criterion for existence of a relative logarithmic connection on $E$ singular along $Y/T$?
\end{question}

For each fiber $\pi^{-1}(s) = X_s$, $s\in S$, we set $Y_s:= X_s \cap Y$. Then $X_s$ intersect $Y$ transversally whenever $Y_s = X_s \cap Y \neq \emptyset$. Also $Y_s = \emptyset$ if and only if $s \notin T$.

 In order to answer the question, we have studied relative logarithmic connection and relative logarithmic Atiyah bundle in section \ref{pre}, and we observe the following:
\begin{proposition}
[Proposition \ref{pro:3J}]
Let $\pi\,:\, X \,\longrightarrow\, S$ be a surjective holomorphic proper
submersion with connected fibers and $\varpi\,:\, E \,\longrightarrow\, X$ a 
holomorphic vector bundle. 
Let $Y/T$ is a relative SNC divisor over $X$.
Let $\nabla$ be a relative logarithmic connection 
on $E$. Then for every $s \,\in\, T$, we have a logarithmic connection $\nabla_s$
on the holomorphic vector bundle $E_s\,\longrightarrow\, X_s$.  In other words, we have a
family $\{\nabla_s\,\mid\, s \in S\}$ which consists of  logarithmic and holomorphic connections on the holomorphic 
family of vector bundles $\{E_s\,\longrightarrow\, X_s\,\mid\, s\in S\}$ depending on $s \in T$ and $s \notin T$.
\end{proposition} 
%In section 3 we prove the following:
%\begin{theorem}
%(Theorem \ref{thm:1})
% Let $\pi : X \longrightarrow S$ be a surjective holomorphic proper submersion of complex manifolds with connected fibres and $E$ be a holomorphic vector bundle on $X$.  We fix an \rm SNC \it divisor $T$ on $S$ and define $Y:= \pi^{-1}(T)$. Then the followings are equivalent:
% \begin{enumerate}
% \item The exact sequence 
% $$0\,\longrightarrow\, \END[\struct{X}]{E}\,\stackrel{\imath}{\longrightarrow}\, \cat{A}t_S(E)(-\log Y)
%\,\stackrel{\sigma_1}{\longrightarrow}\, 
%\cat{T}_{X/S}(- \log Y) \,\longrightarrow\, 0,$$
% splits holomorphically.
% \item $E$ admits a relative logarithmic connection singular along $Y$.
% \item The extension class $\at[S]{E}(\log Y) \in \coh{1}{X}{\Omega^1_{X/S}(\log Y) \otimes \END[\struct{X}]{E}}$ is zero.
% \end{enumerate}
%\end{theorem}
We also give a sufficient condition for existence of relative logarithmic connection on a holomorphic vector bundle. More specifically we prove the following:
\begin{theorem}
[Theorem \ref{thm:2}]
Let $\pi : X \longrightarrow S$ be a surjective holomorphic proper submersion of complex manifolds with connected fibres and $E$ be a holomorphic vector bundle on $X$.  Let $Y/T$ be a relative SNC divisor over $X$. Suppose that the vector bundle $E_s:=E\vert_{X_s}$ admits a logarithmic connection singular along $Y_s$ for each $s\in S$, and $$ \coh{1}{S}{ \pi_*(\Omega^1_{X/S}(\log Y) \otimes \END[\struct{X}]{E})} = 0.$$ Then, $E$ admits a relative logarithmic connection singular along $Y$.
\end{theorem}
In the final section, we introduce the notion of relative residue, and motivated by a result due to Ohtsuki \cite[Theorem 3]{O}, we prove the following result in the relative context. For the notations in the following theorem see section \ref{rel_residue}.
\begin{theorem}
[Theorem \ref{thm:3}]
Let $\pi : X \longrightarrow S$ be a surjective holomorphic proper submersion of complex manifolds with connected fibres and $E$ be a holomorphic vector bundle on $X$. Assume that $X$ is compact. Let $Y/T$ be  a relative SNC divisor over $X$.
Let $D$ be a relative logarithmic connection on $E$
singular over $Y/T$. Then, we have following relation in 
$\drcoh{2k}{dR}{X/S}(S)$
\begin{equation}
\label{eq:13.1}
C^S_k(E) = (-1)^k \{ \sum_{I \in J^k} \sum_{\alpha} \Res[X/S]{D}{Y_{I^*}}^{k, \alpha} C^p_S(Y^\alpha_{I^*})\} \prod_{m = 1}^p C^1_S(Y_{i^*_m})^{a_m -1},
\end{equation}
where $C^S_k(E)$ denote the $k$-th relative Chern class of $E$.
\end{theorem}

\section{Preliminaries}
\label{pre}

\subsection{Logarithmic forms}
\label{log_form}
Let $X$ be a connected smooth complex manifold of dimension at least 1. An effective divisor $D$ on $X$ is said to be a simple normal crossing or SNC in short, if $D$ is reduced, each irreducible component of $D$ is smooth, and for each point $x\in X$, there exists a local system $\bigl(U,z_1,z_2,\cdots,z_n\bigr)$ around $x \in U \subset X$ such that $D \cap U$ is given by the equation $z_1z_2\cdots z_r = 0$ for some integer $r$ with $1\leq r \leq n$. This means that the irreducible components of $D$ passing through $x$ are given by the equations $z_i=0$ for $i=1,2,\cdots,r$, and the these components intersect each other transversally.

For an integer $k\geq 0$ and for an SNC divisor $D$ on $X$, a section of $$\Omega^k_X( D) := \Omega^k_X\otimes_{\mathcal{O}_X}\mathcal{O}_X(D)$$ is called a meromorphic $k$-form on $X$. A meromorphic $k$-form $\alpha \in \Omega^k_X(\log D)(U)$ on an open set $U\subset X$ is said to have logarithmic pole along $D$ if it satisfies the following conditions:
\begin{enumerate}
\item \label{a}$\alpha$ is holomorphic on $U \setminus (U\cap D)$ and $\alpha$ has pole of order at most one along each irreducible component of $D$.

\item \label{b}
The condition \eqref{a} should also holds for $d\alpha$, where $d$ is the holomorphic exterior differential operator.
\end{enumerate}

We denote the sheaf of meromorphic $k$-forms on $X$ having logarithmic pole along $D$ by $\Omega^k_X(\log D)$, and call it sheaf of {\bf logarithmic} $k$-{\bf forms} on $X$
singular over $D$.

\subsection{Complex analytic families}
\label{CAF}
Let $(S,\struct{S})$ be a complex manifold of dimension $n$.  For each 
 $t \in S$, let there be given a compact connected complex manifold $X_t$ 
 of fixed dimension $l$. We say that the set $\{X_t : t \in S \}$ of compact
 connected complex manifolds is called a  \emph{complex analytic family of
 compact connected complex manifolds}, if there is a complex manifold 
 $(X, \struct{X})$ and a surjective holomorphic map $ \pi: X \to S $ of
 complex manifolds such that the followings hold;
 \begin{enumerate}
 \item \label{a1} $ \pi^{-1}(t) = X_t$, for all $t \in S$,
 \item \label{a2} $ \pi^{-1}(t)$ is a compact connected complex 
 submanifold of $X$, for all $t \in S$,
 \item \label{a3} the rank of the Jacobian matrix of $\pi$ is equal to 
 $n$ at each point of $X$. 
 \end{enumerate}
 
Note that, if such a $\pi$ exists, then $\pi : X \longrightarrow S$ is a surjective holomorphic proper submersion such that each fiber $\pi^{-1}(s) = X_s$ is connected for every $s\in S$.\\
Let $d\pi_S : TX \longrightarrow \pi^*T_S$ be the differential of $\pi$.
Then the sheaf of holomorphic sections of the subbundle 
$T(X/S):=\Ker{d\pi_S} \subset TX$
is called the relative tangent sheaf of $\pi$, denoted by $\cat{T}_{X/S}$. 

We have the following short exact sequence
$$0\longrightarrow \cat{T}_{X/S}\longrightarrow \cat{T}_X \xrightarrow{d\pi_S} \pi^*\cat{T}_S\longrightarrow 0.$$
of locally free $\struct{X}$-modules.

The dual $\cat{T}_{X/S}^*$ is called the relative cotangent sheaf of $\pi$ and it is denoted by $\Omega^1_{X/S}$. Dualizing the above short exact sequence we get
$$0\longrightarrow \pi^*\Omega^1_{S}\longrightarrow \Omega^1_X\longrightarrow \Omega^1_{X/S}\longrightarrow 0.$$

Note that both the relative tangent sheaf $\cat{T}_{X/S}$ and the relative cotangent sheaf $\Omega^1_{X/S}$ are locally free $\mathcal{O}_X$-modules of rank $l$.

\subsection{Relative logarithmic connection and Atiyah Bundle} 
The notion of relative logarithmic connection  was  introduced by P. Deligne 
in \cite{D}. For more details on logarithmic and meromorphic connections we refer \cite{D}, \cite{BM}. We recall the definition of relative logarithmic connection on a holomorphic vector bundle.

Let $E$ be a holomorphic vector bundle of rank $r$ over $X$.
A {\bf relative logarithmic connection} on $E$ singular along $Y$ is a first order holomorphic differential operator 
 $$D:E\longrightarrow E\otimes \Omega^1_{X/S}(\log Y)$$
 which satifies the Leibniz property
 \begin{equation*}
 D(f s) = f D(s) + \text{d}_{X/S}(f) \otimes s
 \end{equation*}
 where $s$ and  $f$ are local sections of $E$ and  $\struct{X}$, respectively.

  In \cite[section 2]{BS}, the notions of $S$-derivation, $S$-connection and $S$-differential operators have been introduced in the relative set up. 
  
  For a proper submersion $\pi : X \longrightarrow S$ as above, and for a vector bundle $E$ on $X$, we recall the following symbol exact sequence from \cite[Proposition 4.2]{BS},
$$
0\, \longrightarrow\, \END[\struct{X}]{E}\,\stackrel{\imath}{\longrightarrow}\, 
\DIFF[1]{S}{E}{E}\,\stackrel{\sigma_1}{\longrightarrow}\,
\cat{T}_{X/S} \otimes \END[\struct{X}]{E} \,\longrightarrow\, 0\, ,
$$
where $\sigma_1$ is the symbol morphism, and $\DIFF[1]{S}{E}{E}$ is the sheaf of first order $S$-differential operators on $E$.
Define a bundle $$\cat{A}t_S(E) := \sigma_1^{-1}(\cat{T}_{X/S} \otimes \id{E}),$$
which is known as {\bf relative Atiyah bundle} of $E$ 
and fits in to  the following {\bf Atiyah  exact sequence}
\begin{equation}\label{eq:2}
0\,\longrightarrow\, \END[\struct{X}]{E}\,\stackrel{\imath}{\longrightarrow}\, \cat{A}t_S(E)
\,\stackrel{\sigma_1}{\longrightarrow}\, 
\cat{T}_{X/S} \,\longrightarrow\, 0.
\end{equation}

Further, 
we define 
$$\cat{A}t_S(E)(-\log Y):= \sigma_1^{-1}( \id{E} \otimes \cat{T}_{X/S} \otimes \struct{X}(-Y)).$$
So, we have the following short exact sequence 
\begin{equation}
\label{eq:3}
0\,\longrightarrow\, \END[\struct{X}]{E}\,\stackrel{\imath}{\longrightarrow}\, \cat{A}t_S(E)(-\log Y)
\,\stackrel{\sigma_1}{\longrightarrow}\, 
\cat{T}_{X/S}(- \log Y) \,\longrightarrow\, 0,
\end{equation}
which we call {\bf relative logarithmic Atiyah exact sequence.}

The extension class of the logarithmic Atiyah exact sequence \eqref{eq:3} of a holomorphic 
vector bundle $E$ over $X$ is an element of cohomology
group
$$\coh{1}{X}{\HOM[\struct{X}]{\cat{T}_{X/S}(-\log Y)}{\END[\struct{X}]{E}}}.$$ This extension
class is called the {\bf relative logarithmic Atiyah class} of $E$, and it is denoted by $\at[S]{E}(\log Y)$. 
Note that
$$
\coh{1}{X}{\HOM[\struct{X}]{\cat{T}_{X/S}(-\log Y)}{\END[\struct{X}]{E}}}\,\cong \,
\coh{1}{X}{\Omega^1_{X/S}(\log Y) \otimes \END[\struct{X}]{E}}\, ,
$$
therefore, we have 
$$\at[S]{E}(\log Y) \in \coh{1}{X}{\Omega^1_{X/S}(\log Y) \otimes \END[\struct{X}]{E}}.$$

\subsection{Family of logarithmic connections}
Now, we describe that given a relative logarithmic connection gives a family of logarithmic connections.

Let $\varpi\,:\,E \,\longrightarrow\, X$ be a holomorphic vector bundle. For every $s \,\in\, S$,
the restriction of $E$ to $X_s \,=\, \pi^{-1}(s)$ is denoted by $E_s$.
Let $U$ be an open subset of $X$ and $\alpha\, :\,U \,\longrightarrow\, E$ a holomorphic section.
We denote by $r_s(\alpha)$ the restriction of $\alpha$ to $X_s\cap U$, whenever $U \cap
X_s\,\neq\, \emptyset$. Clearly, $r_s(\alpha)$ is a holomorphic section of $E_s$ over
$U \cap X_s$. The map $r_s \,:\, \alpha\,\longmapsto\, r_s(\alpha)$ induces, therefore, a 
homomorphism of $\C$-vector spaces from $E$ to $E_s$, which is denoted
by the same symbol $r_s$. Also, $X_s$ is a complex 
submanifold of $X$, so $\struct{X}|_{X_s}\,=\, \struct{X_s}$. We also have
the restriction map $r_s\,:\,\END[\struct{X}]{E}\,\longrightarrow\, \END[\struct{X_s}]{E_s}$.

Similarly, if $P\,:\, E\,\longrightarrow\, F$ is a first order $S$-differential operator,
where $F$ is a holomorphic vector bundle over $X$, then
the restriction map $r_s\,:\,E_s \,\longrightarrow\, F_s$ gives rise to a first order 
differential operator $P_s\,:\, E_s\,\longrightarrow\, F_s$ for every $s\, \in\, S$. Thus,
we have the restriction map $r_s\,:\,\DIFF[1]{S}{E}{F}\,\longrightarrow\, \DIFF[1]{\C}{E_s}{F_s}$.
In particular, for $E\,=\, F$, we have the restriction map 
$r_s\,:\,\DIFF[1]{S}{E}{E} \,\longrightarrow\, \DIFF[1]{\C}{E_s}{E_s}$ for every $s\,\in \,S$.
Since, the restriction of the relative tangent bundle $T(X/S)$ to each
fiber $X_s$ of $\pi$ is the tangent bundle $T(X_s)$ of the fiber
$X_s$, we have the restriction map $r_s\,:\,\cat{T}_{X/S}(- \log Y)\,\longrightarrow\, \cat{T}_{X_s}(- \log Y_s)$. 
 
Now, for every $s\,\in\, S$, the restriction maps gives a commutative diagram
\begin{equation}
\label{eq:cd1}
\xymatrix@C=2em{
0 \ar[r] & \END[\struct{X}]{E} \ar[d]^{r_s} \ar[r] & \cat{A}t_S(E)(-\log Y) 
\ar[d]^{r_s} \ar[r]^{\sigma_1} & \cat{T}_{X/S}(- \log Y) \ar[d]^{r_s} \ar[r] & 0 \\
0 \ar[r] & \END[\struct{X_s}]{E_s} \ar[r] & \cat{A}t(E_s)(- \log Y_s) \ar[r]^{\sigma_{1s}} & \cat{T}_{X_s}(- \log Y_s) \ar[r] & 0 
}
\end{equation}
where the bottom sequence is the logarithmic Atiyah sequence of the holomorphic vector
bundle $E_s$ over $X_s$  singular along $Y_s$ and $\sigma_{1s}$ is the restriction
of the symbol map $\sigma_1$.

Suppose that $E$ admits a relative logarithmic connection, which is 
equivalent to saying that the relative logarithmic Atiyah sequence in \eqref{eq:3} splits 
holomorphically. If $$\nabla\,:\,\cat{T}_{X/S}(- \log Y)\,\longrightarrow\,\cat{A}t_S(E)(- \log Y)$$ is 
a holomorphic splitting of the relative logarithmic Atiyah sequence in \eqref{eq:3},
then for every $s \,\in\, T$, the restriction of
$\nabla$ to $\cat{T}_{X_s}(- \log Y_s)$ gives an $\struct{X_s}$-module homomorphism
$$\nabla_s\,:\, \cat{T}_{X_s}(- \log Y_s) \,\longrightarrow\, \cat{A}t(E_s)(- \log Y_s).$$
Now, $\nabla_s$ is a 
holomorphic splitting of the logarithmic  Atiyah sequence of the holomorphic vector 
bundle $E_s$, which follows from the fact that the restriction maps $r_s$
defined above are surjective. 
Note that if $s \notin T$, then $\nabla_s$ is nothing 
but the holomorphic connection in $E_s$, because $Y_s = \emptyset$.

Thus, we have the following:
\begin{proposition}\label{pro:3J}
Let $\pi\,:\, X \,\longrightarrow\, S$ be a surjective holomorphic proper
submersion with connected fibers and $\varpi\,:\, E \,\longrightarrow\, X$ a 
holomorphic vector bundle. 
Let $Y/T$ be the relative SNC divisor over $X$.
Let $\nabla$ be a relative logarithmic connection 
on $E$. Then for every $s \,\in\, T$, we have a logarithmic connection $\nabla_s$
on the holomorphic vector bundle $E_s\,\longrightarrow\, X_s$, . In other words, we have a
family $\{\nabla_s\,\mid\, s \in S\}$ which consists of  logarithmic and holomorphic connections on the holomorphic 
family of vector bundles $\{E_s\,\longrightarrow\, X_s\,\mid\, s\in S\}$ depending on $s \in T$ and $s \notin T$.
\end{proposition}

\section{A sufficient condition for existence of logarithmic connections}
In this section, we prove the equivalent assertions for a holomorphic vector bundle to admit a relative logarithmic connections. Further, we give a sufficient condition for existence of relative logarithmic connections. 
\begin{theorem}
\label{thm:1}
 Let $\pi : X \longrightarrow S$ be a surjective holomorphic proper submersion of complex manifolds with connected fibres and $E$ be a holomorphic vector bundle on $X$.  Let $Y/T$ be the relative SNC divisor over $X$. Then the followings are equivalent:
 \begin{enumerate}
 \item The exact sequence 
 $$0\,\longrightarrow\, \END[\struct{X}]{E}\,\stackrel{\imath}{\longrightarrow}\, \cat{A}t_S(E)(-\log Y)
\,\stackrel{\sigma_1}{\longrightarrow}\, 
\cat{T}_{X/S}(- \log Y) \,\longrightarrow\, 0,$$
 splits holomorphically.
 \item $E$ admits a relative logarithmic connection singular along $Y$.
 \item The extension class $\at[S]{E}(\log Y) \in \coh{1}{X}{\Omega^1_{X/S}(\log Y) \otimes \END[\struct{X}]{E}}$ is zero.
 \end{enumerate}
\end{theorem}

\begin{proof}
  $(i) \iff (ii)$ Suppose the Atiyah exact sequence splits holomorphically, i.e. there exists an $\mathcal{O}_X$-module homomophism $$\nabla : \cat{T}_{X/S}(- \log Y) \longrightarrow \cat{A}t_S(E)(-\log Y)$$
  such that $\sigma_1 \circ \nabla = \id{\cat{T}_{X/S}(- \log Y)}$. For any open set $U\subset X$, for every $\xi \in \cat{T}_{X/S}(- \log Y)(U)$ and for every $a \in \mathcal{O}_X(-Y)(U)$, we then have
  $$ \sigma_1(\nabla_U(\xi))(a) = [ \nabla_U(\xi),a ] = \xi(a) \id{E},$$
  see \cite[Proposition 3.1]{BS} for the symbol map $\sigma_1$.
  This in particular implies that $$\nabla_U(\xi)(as) = a \nabla_U(\xi)(s)+\xi(a)s.$$
  This shows that $\nabla$ satisfies the Leibniz condition. Since $\cat{A}t_S(E)(-\log Y)$ is an $\mathcal{O}_X$ submodule of $\END[\struct{S}]{E}(-\log Y)$, we conclude that $\nabla$ indeed defines a relative logarithmic connection singular along $Y$.
  
  Conversely, any relative logarithmic connection singular along $Y$ satisfies Leibniz property. 
  In particular, this will give a holomorphic splitting of the Atiyah exact sequence.
  
  $(i)\iff (iii)$ The splitting of the exact sequence 
 $$0\,\longrightarrow\, \END[\struct{X}]{E}\,\stackrel{\imath}{\longrightarrow}\, \cat{A}t_S(E)(-\log Y)
\,\stackrel{\sigma_1}{\longrightarrow}\, 
\cat{T}_{X/S}(- \log Y) \,\longrightarrow\, 0,$$
 is given by the vanishing of the extension class $$\at[S]{E}(\log Y) \in \text{Ext}^1 (\cat{T}_{X/S}(-\log Y), \END[\struct{X}]{E}).$$
 
 Note that $$\text{Ext}^1 (\cat{T}_{X/S}(-\log Y), \END[\struct{X}]{E}) = \coh{1}{X}{\Omega^1_{X/S}(\log Y) \otimes \END[\struct{X}]{E}}. $$
 This proves the theorem.
\end{proof}

\begin{theorem}
\label{thm:2}
Let $\pi : X \longrightarrow S$ be a surjective holomorphic proper submersion of complex manifolds with connected fibres and $E$ be a holomorphic vector bundle on $X$.  Let $Y/T$ be the relative SNC divisor over $X$. Suppose that the vector bundle $E_s:=E\vert_{X_s}$ admits a logarithmic connection singular along $Y_s$ for each $s\in S$, and $$ \coh{1}{S}{ \pi_*(\Omega^1_{X/S}(\log Y) \otimes \END[\struct{X}]{E})} = 0.$$ Then, $E$ admits a relative logarithmic connection singular along $Y$.
\end{theorem}

\begin{proof}
Consider the relative logarithmic  Atiyah exact sequence in \eqref{eq:3}. Now, tensoring it by
$\Omega^1_{X/S}(\log Y)$ gives  the following exact sequence
\begin{equation*}
0\,\longrightarrow\, \Omega^1_{X/S}(\log Y)(\END[\struct{X}]{E})\,\longrightarrow\, \Omega^1_{X/S}(\log Y)(\cat{A}t_S(E)(- \log Y))
\end{equation*}
\begin{equation}\label{q}
\,\stackrel{q}{\longrightarrow}\,\Omega^1_{X/S}(\log Y)\otimes \cat{T}_{X/S}(- \log Y)
\,\longrightarrow\, 0\, .
\end{equation}

We have  $\struct{X}\cdot \text{Id} \, \subset\, \text{End}(\cat{T}_{X/S}(- \log Y))\,=\,
\Omega^1_{X/S}(\log Y)\otimes \cat{T}_{X/S}(- \log Y)$. Define
$$
\Omega^1_{X/S}(\log Y)(\cat{A}t'_S(E))\, :=\, q^{-1}(\struct{X}\cdot \text{Id})\, \subset\,
\Omega^1_{X/S}(\log Y)(\cat{A}t_S(E)(- \log Y))\, ,
$$
where $q$ is the projection in \eqref{q}. So we have the short exact sequence of sheaves
\begin{equation}\label{q2}
0\,\longrightarrow\, \Omega^1_{X/S}(\log Y)(\END[\struct{X}]{E})\,\longrightarrow\, \Omega^1_{X/S}(\log Y)(\cat{A}t'_S(E))
\,\stackrel{q}{\longrightarrow}\,\struct{X}\,\longrightarrow\, 0
\end{equation}
on $X$, where $\Omega^1_{X/S}(\log Y)(\cat{A}t'_S(E))$ is constructed above. Let
\begin{equation}\label{q3}
\Phi\, :\, {\mathbb C}\, \subset \, {\rm H}^0(X,\, \struct{X}\cdot \text{Id})\, \longrightarrow\,
{\rm H}^1(X,\, \Omega^1_{X/S}(\log Y)(\END[\struct{X}]{E}))
\end{equation}
be the homomorphism in the long exact sequence of cohomologies associated to the exact sequence
in \eqref{q2}. The relative Atiyah class $\at[S]{E}(\log Y)$ (see Theorem \ref{thm:1}) coincides
with $\Phi(1)\, \in\, {\rm H}^1(X,\, \Omega^1_{X/S}(\log Y)(\END[\struct{X}]{E}))$. Therefore, from
Theorem \ref{thm:1}, it follows that 
$E$ admits a relative logarithmic connection if and only if
\begin{equation}\label{q4}
\Phi(1)\, =\,0\, .
\end{equation}

Note that
${\rm H}^1(X,\, \Omega^1_{X/S}(\log Y)(\END[\struct{X}]{E}))$ fits in the following  exact sequence
\begin{equation*}
{\rm H}^1(S,\, \pi_*(\Omega^1_{X/S}(\log Y)(\END[\struct{X}]{E})))\, \stackrel{\beta_1}{\longrightarrow}\,
{\rm H}^1(X,\, \Omega^1_{X/S}(\log Y)(\END[\struct{X}]{E}))\,
\end{equation*}
\begin{equation}\label{q5}
\stackrel{q_1}{\longrightarrow}\,
{\rm H}^0(S,\, R^1\pi_*(\Omega^1_{X/S}(\log Y)(\END[\struct{X}]{E})))\, ,
\end{equation}
where $\pi$ is the projection of $X$ to $S$.

The given condition that for every
$s \in S$, there is a logarithmic connection on the holomorphic vector bundle
$\varpi|_{E_s}\, :\, E_s \,\longrightarrow\, X_s$, implies that
$$
q_1(\Phi(1))\,=\, 0\, ,
$$
where $q_1$ is the homomorphism in \eqref{q5}. Therefore, from the exact sequence in
\eqref{q5} we conclude that
$$
\Phi(1)\, \in\, \beta_1({\rm H}^1(S,\, \pi_*(\Omega^1_{X/S}(\log Y)(\END[\struct{X}]{E}))))\, .
$$
Finally, the given condition that ${\rm H}^1(S,\, \pi_*(\Omega^1_{X/S}(\log Y)(\END[\struct{X}]{E})))\,=\, 0$.
implies that $\Phi(1)\, =\, 0$. Since \eqref{q4} holds, the vector bundle
$E$ admits a relative logarithmic connection.
\end{proof}

\section{Relative Chern classes in terms of relative residue}
\label{rel_residue}
In this section, we express the relative Chern classes 
in terms of relative residues which generalises \cite[Theorem 3]{O} due to Ohtsuki in the relative context.
\subsection{Relative residue:}
We define the relative residues of a relative 
logarithmic connection $D$ on $E$. 
Let $$Y=\bigcup_{j \in J} Y_j$$ be the decomposition of 
$Y$ into it's irreducible components, and 
$$\tau_j:Y_j\longrightarrow X$$ the inclusion map for every $j \in J$. Since $Y$ is a normal crossing divisor on $X$, we can choose a fine open cover $\{ U_{\lambda}:\lambda \in \Lambda \}$ of $X$ such that for each $\lambda \in \Lambda$, we have the following:
\begin{enumerate}
\item each $E\vert_{U_{\lambda}}$ is trivial,
\item for each irreducible component $Y_j$ of $Y$ with $Y_j\cap U_{\lambda} \neq \emptyset$, we can choose a local coordinate function $f_{\lambda j} \in \mathcal{O}_X(U_{\lambda})$ for a local coordinate system on $U_{\lambda}$, such that $f_{\lambda j}$ is a defining equation of $Y_j\cap U_{\lambda}$. If $Y_j\cap U_{\lambda} = \emptyset$, then we take $f_{\lambda j} = 1$.
\end{enumerate}

Let $e_{\lambda} = (e_{1 \lambda}, \ldots, e_{r \lambda})$ be the local frame of $E$, and  $\omega_{\lambda}$  the relative connection matrix of
$D$ with respect to a holomorphic local frame $e_{\lambda}$ for $E$ on $U_{\lambda}$, that is, we have
$$D(e_{\lambda}) = \omega_{\lambda}\otimes e_{\lambda},$$
where $\omega_{\lambda}$ is the $r\times r$ matrix whose entries are holomorphic sections of $\Omega^1_{X/S}(\log Y)$ over $U_{\lambda}$. For each $Y_j$, the matrix $\omega_{\lambda}$ can be written as 
$$\omega_{\lambda} = R_{\lambda j} \frac{df_{\lambda j}}{f_{\lambda j}} + S_{\lambda j},$$
where $R_{\lambda j}$ is an $r\times r$ matrix with entries in $\mathcal{O}_X(U_{\lambda})$ and $S_{\lambda j}$ is a $r\times r$ matrix with entries in $\bigl(\Omega^1_{X/S}(\log Y)\bigr)(U_{\lambda})$ with simple pole along $\bigcup\limits_{i\neq j}Y_i$.

Then 
$$\text{Res}_{X/S}(\omega_{\lambda},Y_j) := R_{\lambda j}$$
is an $r \times r$ matrix whose entries are holomorphic functions on $U_\lambda \cap Y_j$
and it is independent of
choice of local defning equation $f_{\lambda j}$
 for $Y_j$.
Then $\{\text{Res}_{X/S}(\omega_{\lambda},Y_j)\}_{\lambda \in \Lambda}$ defnes a holomorphic
global section
\begin{equation}
\label{eq:5}
\text{Res}_{X/S}(D, Y_j) \in \coh{0}{Y_j}{\END[\struct{X}]{E} \vert_{Y_j}}
\end{equation}
called the \textbf{relative residue} of the relative connection $D$ along $Y_j$.

For every $s \in T$, we have the decomposition 
$$Y_s = Y \cap \pi^{-1}(s) = \bigcup_{j \in J} (Y_j \cap \pi^{-1}(s))$$
of $Y_s$ into its irreducible components.
We denote $Y_j \cap \pi^{-1}(s)$ by $Y_{js}$. 

Recall that for a given relative logarithmic connection 
$D$ on $E$ singular over $Y/T$, we get a logarithmic connections 
$D_s$ on $E_s := E \vert_{X_s}$ singular over $Y_s$
for every $s \in T$.
Then we have residue (see \cite{O}) of $D_s$ over each irreducible component $Y_{js} $ of $Y_s$ denoted as 
\begin{equation}
\label{eq:7.5}
\Res[X_s]{D_s}{Y_{js}} \in \coh{0}{Y_{js}}{\END[\struct{X_s}]{E_s} \vert_{Y_{js}}},
\end{equation}
for every $s \in T$.

\subsection{Relative Chern class}

We recall the  definition of the relative Chern classes
of a holomorphic vector bundle over $\pi : X \to S$, for more details see \cite[Section 4.6]{BS}.

Let $E$ be a hermitian holomorphic vector bundle on $X$, 
that is, $E$ is a holomorphic vector bundle with Hermitian 
metric on it. Then there exists  canonical (smooth)
connection $\nabla$ on $E$ compatible with the Hermitian 
metric. 

Let $\cat{A}^r_{X/S}$ denote the sheaf of \emph{complex
valued smooth relative $r$-form} on $X$ over $S$.
Then we have relative de Rham complex 
$$
0\,\longrightarrow\,\pi^{-1}\cat{C}^\infty_{S}\,\longrightarrow\,\cat{C}^\infty_{X} \xrightarrow{d_{X/S}}
\cat{A}^1_{X/S} \xrightarrow{d_{X/S}} \cdots \xrightarrow{d_{X/S}}
\cat{A}^{2l}_{X/S} \,\longrightarrow\, 0
$$
of $\cat{C}^\infty_{X}$-module and $S$-linear maps, which we denote by the pair $(\cat{A}^\bullet_{X/S}, d_{X/S})$.

Moreover, because of the following short exact sequence
$$0\longrightarrow \pi^*\cat{A}^1_{S}\longrightarrow \cat{A}^1_X\longrightarrow \cat{A}^1_{X/S}\longrightarrow 0,$$
we get a relative smooth connection $D$ on $E$ induced from $\nabla$.

Given a relative smooth connection  $D$ on $E$. Let $(U_\alpha,\,h_\alpha)$ be a trivialization
of $E$ over $U_\alpha\,\subset\, X$.
Let $R$ be the $S$-curvature (relative curvature)  form for $D$, and let $\Omega_\alpha\,=\, (\Omega_{ij\alpha})$ 
be the curvature matrix of $D$ over $U_\alpha$,  so
$\Omega_{ij\alpha}\,\in\,\cat{A}^2_{X/S}(U_\alpha)$. We have 
$\Omega_\beta\,=\, g_{\alpha \beta}^{-1} \Omega_{\alpha} g_{\alpha \beta}$, 
where $g_{\alpha \beta}\,:\, U_{\alpha} \cap U_{\beta}\,\longrightarrow\, {\rm GL}_r(\C)$ is the 
change of frame matrix (transition function), which is a smooth map.

Consider the adjoint action of ${\rm GL}_r(\C)$ on it Lie algebra $\mathfrak{gl}_r(\C)\,=\,
{\rm M}_r(\C)$. Let $f$ be a ${\rm GL}_r(\C)$-invariant homogeneous 
polynomial on $\mathfrak{gl}_r(\C)$ of degree $p$. Then, we can associate a unique
$p$-multilinear symmetric map $\widetilde{f}$ on $\mathfrak{gl}_r(\C)$ 
such that $f(X)\,=\, \widetilde{f}(X,\cdots ,X)$, for all $X \in \mathfrak{gl}_r(\C)$.
Define $$\gamma_\alpha \,=\, \widetilde{f}(\Omega_\alpha,\,\cdots, \,\Omega_\alpha)
\,=\, f(\Omega_\alpha)\,\in\,
\cat{A}^{2p}_{X/S}(U_\alpha)\, .$$
Since $f$ is ${\rm GL}_r(\C)$-invariant, it follows that $\gamma_\alpha$ is independent of
the choice of frame, and hence it
defines a global smooth relative differential form of degree $2p$, which we denote by the symbol
$\gamma \in \cat{A}^{2p}_{X/S}(X)$.

\begin{theorem}\cite[Theorem 4.9]{BS}
\label{thm:5}
Let $\pi\,:\, X\,\longrightarrow\,S$ be a surjective holomorphic proper 
submersion of complex manifolds with connected fibers and $\varpi\,:\, E \,\longrightarrow\, X$ a differentiable
family of complex vector bundle. Let $D$ be a relative smooth connection on $E$.
Suppose that $f$ is a ${\rm GL}_r(\C)$-invariant 
polynomial function on $\mathfrak{gl}_r(\C)$ of degree $p$. Then the following hold:
\begin{enumerate}
\item $\gamma\,=\, f(\Omega_\alpha)$ is $d_{X/S}$-closed, that is,
$d_{X/S}(\gamma) \,=\, 0$.

\item The image $[\gamma]$ of $\gamma$ in the relative de Rham cohomology group
$$\gcoh{2p}{\Gamma(X,\,\cat{A}^\bullet_{X/S} )} \,=\, 
\coh{2p}{X}{\pi^{-1}\cat{C}^\infty_{S}}$$
is independent of the relative smooth connection $D$ on $E$.
\end{enumerate}
\end{theorem}

Define homogeneous polynomials $f_p$ on $\mathfrak{gl}_r(\C)$, of degree 
$p\,=\, 1,\,2,\,\cdots ,\,r$, to be the coefficient of $\lambda^p$ in the following 
expression:
\begin{equation}
\label{eq:9}
\mathrm{det}(\lambda \mathrm{I} + \frac{\sqrt{-1}}{2 \pi} A) \,=\, \Sigma_{j=0}^{r}
\lambda^{r-j} f_j(\frac{\sqrt{-1}}{2 \pi} A)\, ,
\end{equation}
where $f_0(\frac{\sqrt{-1}}{2 \pi} A) \,=\, 1$ while $f_r(\frac{\sqrt{-1}}{2 \pi} A)$ is the 
coefficient of $\lambda^0$. These polynomials $f_1,\,\cdots,\, f_r$ are $GL_r(\C)$-invariant, and
they generate the algebra of ${\rm GL}_r(\C)$-invariant polynomials on 
$\mathfrak{gl}_r(\C)$. We now define the \emph{p-th cohomology class} as follows:
$$
c^{S}_p(E)\,=\, [f_p(\frac{\sqrt{-1}}{2 \pi} \Omega)]\,\in\, \gcoh{2p}{\Gamma(X,\,
\cat{A}^\bullet_{X/S})}
$$
for $p \,=\,0,\,1,\,\cdots ,\,r$. 
 
The relative de Rham cohomology sheaf $\drcoh{p}{dR}{X/S} \,\cong\, 
R^p \pi_\ast(\pi^{-1}\cat{C}^\infty_{S})$ on $S$ is by 
definition the sheafification of the presheaf $V\,\longmapsto\, 
\coh{p}{\pi^{-1}(V)}{\pi^{-1}\cat{C}^\infty_{S}|_{\pi^{-1}(V)}}$,
for open subset $V \,\subset\, S$. Therefore, we have a natural homomorphism
\begin{equation}
\label{eq:8}
\rho\,:\, \coh{2p}{X}{\pi^{-1}\cat{C}^\infty_{S}} \,\longrightarrow\,
\drcoh{2p}{dR}{X/S}(S)
\end{equation}
which maps $c^{S}_p(E)$ to $\rho( c^S_p(E))\,\in\, \drcoh{2p}{dR}{X/S}(S)$.

Define $C^S_p(E) \,=\, \rho(c^S_p(E))$. We call $C^S_p(E)$ the \emph{p-th 
relative Chern class of E over S}. 
Let $$C^S(E)\,=\, \sum_{p \geq 0} C^S_p(E)\,\in\, \drcoh{*}{dR}{X/S}(S)\,=\, 
\oplus_{k \geq 0} \drcoh{k}{dR}{X/S}(S)$$ be the \emph{total relative 
Chern class} of $E$.

\subsection{Relative Chern classes in terms of relative residue}
We follow the notations as above.
Let $J^k := J \times \cdots \times J$ be the $k$-fold product of $J$. Let $I = (i_1, \ldots, i_k) \in J^k$.
If there are $p$-different indices among $i_1, \ldots, i_k$, we denote them by $i^*_1, \ldots, i^*_p$, tuple is denoted by
$I^* = (i^*_1, \ldots, i^*_p)$.
Let $a_m$ be the number of $i^*_m$ appearing in $I$,
then we have 
$$\sum_{m =1}^p a_m = k.$$
For given $I \in J^k$, we define 
\begin{equation}
\label{eq:10}
Y_{I^*} = \bigcap_{m = 1}^p Y_{i^*_m}.
\end{equation}
Then either $Y_{I^*} = \emptyset$ or a submanifold of 
$X$ of codimension $p$.
Further, $Y_{I^*}$ need not be connected.
Let 
\begin{equation}
\label{eq:11}
Y_{I^*} = \bigcup_{\alpha} Y^{\alpha}_{I^*}
\end{equation}
be the disjoint union of connected components of $Y_{I^*}$. Then each $Y^\alpha_{I^*}$ is a submanifold of codimension $p$.

Let $\widetilde{f_k}$ be  the unique
$k$-multilinear symmetric map  on $\mathfrak{gl}_r(\C)$ 
such that $$f_k(A)\,=\, \widetilde{f_k}(A,\cdots ,A),$$ for all $A \in \mathfrak{gl}_r(\C)$, where $f_k$ is defined 
in \eqref{eq:9} for every $k = 0, \ldots, r$.
Now onwards we assume that $X$ is compact.

\begin{lemma}
\label{lem:1}
Let $\pi : X \longrightarrow S$ be a surjective holomorphic proper submersion of complex manifolds with connected fibres and $E$ be a holomorphic vector bundle on $X$. Assume that $X$ is compact. Let $Y/T$ be the relative SNC divisor over $X$.
Let $D$ be a relative logarithmic connection on $E$
singular over $Y/T$.
Then for any $I = (i_1, \ldots, i_k) \in J^k$, the following polynomial
$$\widetilde{f_k}(\Res[X/S]{D}{Y_{i_1}}, \Res[X/S]{D}{Y_{}i_2}, \ldots, \Res[X/S]{D}{Y_{i_k}})$$
is constant on each connected component $Y^{\alpha}_{I^*}$ of $Y_{I^*}$ described in \eqref{eq:11}.

\end{lemma}

\begin{proof}
Since $X$ is a compact complex manifold, each connected 
component $Y^{\alpha}_{I^*}$ is a compact complex submanifold  of $X$. Hence proof follows from the fact 
that $\widetilde{f_k}$ is a polynomial function.
\end{proof}

For the simplicity of the notation, we denote the constant number $$\widetilde{f_k}(\Res[X/S]{D}{Y_{i_1}}, \Res[X/S]{D}{Y_{i_2}}, \ldots, \Res[X/S]{D}{Y_{i_k}})$$
on each $Y^\alpha_{I^*}$ in the above Lemma \ref{lem:1}
by $\Res[X/S]{D} {Y_{I^*}}^{k, \alpha}$.

Let $W$ be a submanifold of $X$ of codimension $p$.
Then, we get a cohomology class $[W] \in \coh{2p}{X}{\C}$. 
Because of the following inclusion of sheaves 
$$\C \hookrightarrow \pi^{-1} \cat{C}^{\infty}_S, $$
we get a homomorphism 
\begin{equation}
\label{eq:11.1}
\gamma : \coh{2p}{X}{\C} \longrightarrow \coh{2p}{X}{ \pi^{-1} \cat{C}^{\infty}_S}
\end{equation}
on cohomology groups.
Further using the natural homomorphism 
\begin{equation*}
\rho\,:\, \coh{2p}{X}{\pi^{-1}\cat{C}^\infty_{S}} \,\longrightarrow\,
\drcoh{2p}{dR}{X/S}(S)
\end{equation*}
in \eqref{eq:8}, we define 
\begin{equation}
\label{eq:12}
C^p_S(W) := \rho (\gamma ([W]))
\end{equation}
call it the $p$-th relative Chern classes associated to 
$W$.

\begin{theorem}
\label{thm:3}
Let $\pi : X \longrightarrow S$ be a surjective holomorphic proper submersion of complex manifolds with connected fibres and $E$ be a holomorphic vector bundle on $X$. Assume that $X$ is compact. Let $Y/T$ be the relative SNC divisor over $X$.
Let $D$ be a relative logarithmic connection on $E$
singular over $Y/T$. Then, we have following relation in 
$\drcoh{2k}{dR}{X/S}(S)$
\begin{equation}
\label{eq:13}
C^S_k(E) = (-1)^k \{ \sum_{I \in J^k} \sum_{\alpha} \Res[X/S]{D}{Y_{I^*}}^{k, \alpha} C^p_S(Y^\alpha_{I^*})\} \prod_{m = 1}^p C^1_S(Y_{i^*_m})^{a_m -1},
\end{equation}
where $C^S_k(E)$ denote the $k$-th relative Chern class of $E$.
\end{theorem}
\begin{proof}
It is enough to show the formula \eqref{eq:13} stalkwise, and in particular stalks at $s \in T$. First note that 
for every $s \in S$, and inclusion morphism 
$j : X_s \hookrightarrow X$, we have a natural map (see \cite[Corollary 4.11]{BS} ) $$j^* \,:\,\drcoh{2k}{dR}{X/S}(S)
\,\longrightarrow\,\coh{2k}{X_s}{\C}$$ which maps the $k$-th relative Chern class of $E$ to
the $k$-th Chern class of the vector bundle $E_s \,\longrightarrow\, X_s$, that is,
$j^*(C_k^S(E)) \,=\, c_k(E_s)$, where $c_k(E_s)$ denote the 
$k$-th Chern class of $E_s$.

Note that $\drcoh{2k}{dR}{X/S}$ is a locally free 
$\cat{C}^\infty_{S}$-module, and using the topological proper base change theorem  given in \cite[p.~202, Remark 4.17.1]{G}
and \cite[p.~19, Corollary 2.25]{D},
we have a $\C$-vector space isomorphism
\begin{equation}
\label{eq:3.I.1}
\eta\,:\, \drcoh{2k}{dR}{X/S}_s \otimes_{\cat{C}^\infty_{{S},s}} k(s)
\, \longrightarrow\, \coh{2k}{X_s}{\C}
\end{equation}
for every $s \,\in\, S$. In fact, we have the following commutative diagram;
$$
\xymatrix{\drcoh{2k}{dR}{X/S}(S) \ar[rd]_{j^*} \ar[r] & 
\drcoh{2k}{dR}{X/S}_s \otimes_{\cat{C}^\infty_{{S},s}} k(s) 
\ar[d]^{\eta} \\
& \coh{2k}{X_s}{\C} \\ }
$$
Hence, we get 
\begin{equation}
\label{eq:14}
\eta(C_k^S(E)_s \otimes 1)\,=\, j^*(C_k^S(E)) \,=\, c_k(E_s).
\end{equation}

Let us fix the following notation for $s \in T$;
 $$Y_{s I^*} = Y_{I^*} \cap \pi^{-1}(s) = \bigcup_{\alpha} (Y^{\alpha}_{I^*} \cap \pi^{-1}(s)),$$ 
 $$Y^\alpha_{sI^*} = Y^{\alpha}_{I^*} \cap \pi^{-1}(s),$$ and 
 $$Y_{si^*_m} = Y_{i^*_m} \cap \pi^{-1}(s). $$

Now, note that the  germ at $s \in T$ of the right hand side of the formula \eqref{eq:13} is associated to the logarithmic connection $D_s$ on $E_s \longrightarrow X_s$, that is, we get the following expression 
\begin{equation}
\label{eq:15}
(-1)^k \{ \sum_{I \in J^k} \sum_{\alpha} \Res[X_s]{D_s}{Y_{sI^*}}^{k, \alpha} c_p(Y^\alpha_{sI^*})\} \prod_{m = 1}^p c_1(Y_{si^*_m})^{a_m -1},
\end{equation}

where $\Res[X_s]{D_s}{Y_{sI^*}}^{k, \alpha}$ denote the 
constant function $$\widetilde{f_k}(\Res[X_s]{D_s}{Y_{si_1}}, \Res[X_s]{D_s}{Y_{si_2}}, \ldots, \Res[X_s]{D_s}{Y_{si_k}})$$
on each connected component $Y^\alpha_{sI^*}$  of 
$Y_{s I^*}$.

 From \cite[Theorem 3]{O}, we have 
 \begin{equation}
 \label{eq:16}
 c_k(E_s) = (-1)^k \{ \sum_{I \in J^k} \sum_{\alpha} \Res[X_s]{D_s}{Y_{sI^*}}^{k, \alpha} c_p(Y^\alpha_{sI^*})\} \prod_{m = 1}^p c_1(Y_{si^*_m})^{a_m -1}.
 \end{equation}
 
 In view of  \eqref{eq:14}, \eqref{eq:15}, and \eqref{eq:16}, proof of the theorem is complete.
\end{proof}


\begin{thebibliography} {999}
\bibitem{A} Atiyah, M.F., Complex analytic connections in 
fibre bundles, {\it Trans. Amer. Math.Soc.} {\bf 85}(1957), 181-207.



\bibitem{BM} Borel, A., Grivel, P.-P., Kaup, B., Haefliger, A., 
Malgrange, B., Ehlers, F.,
Algebraic D-modules.
Perspectives in Mathematics, 2. Academic Press, Inc., Boston, 
MA, 1987. xii+355 pp. 

\bibitem{BE} Spencer Bloch and H\'{e}l\`{e}ne Esnault.
\emph{Lectures on algebro-geometric Chern Weil and Cheeger-Chern-Simsons theory for vector bundles,}
The Arithmetic and Geometry of Algebraic Cycles. NATO Science Series, vol 548. Springer, Dordrecht.

\bibitem{BDP} Indranil Biswas, Ananyo Dan and Arjun Paul,
\emph{Criterion for logarithmic connections with prescribed residues,}
Manuscripta Math. 155(1-2) 77-88(2018).


\bibitem{BS} Indranil Biswas and Anoop Singh.
\emph{On the relative connections,}
Communications in Algebra, 48 (2020), no. 4, 1452-1475.

\bibitem{D} Deligne, P., Equations diff\'erentielles \'a points 
singuliers r\'eguliers. Lecture Notes in Mathematics, vol. 163.
Springer, Berlin(1970).

\bibitem{G} R.~Godement, {\it Th\'eorie des faisceaux}, Hermann, Paris, (1964).

\bibitem{O} Ohtsuki, M., A residue formula for Chern classes 
associated
with logarithmic connections. {\it Tokyo J. Math.} Vol.{\bf 5}, 
No. 1, 1982.


\bibitem{W} A.~Weil, G\'en\'eralisation des fonctions ab\'eliennes, {\it J.~Math.~Pures
Appl.} \textbf{17} (1938), 47--87.
\end{thebibliography}
\end{document}